\documentclass[reqno,11pt]{amsart}

\usepackage[all]{xy}
\usepackage{amssymb}
\usepackage{amsgen}
\usepackage{amsmath}
\usepackage{amsthm}
\usepackage{mathrsfs}
\usepackage{cite}
\usepackage{amsfonts}

\newcommand{\Bis}{\mathop{\mathrm{Bis}}\nolimits}

\newcommand{\dom}{\mathop{\boldsymbol d}}
\newcommand{\ran}{\mathop{\boldsymbol r}}

\newcommand{\Sh}{\mathop{\mathrm{Sh}}}

\newcommand{\modu}{\mathrm{mod}\text{-}}

\newcommand{\inv}{^{-1}}

\newcommand{\til}[1]{\ensuremath{\widetilde {#1}}}

\usepackage{xcolor}

\newcommand{\skel}[1]{^{(#1)}}

\newtheorem{Thm}{Theorem}[section]
\newtheorem{Prop}[Thm]{Proposition}

\newtheorem{Lemma}[Thm]{Lemma}
{\theoremstyle{definition}
}
{\theoremstyle{remark}
}
\newtheorem{Cor}[Thm]{Corollary}
{\theoremstyle{remark}
}
{\theoremstyle{remark}
}

{\theoremstyle{remark}
}
{\theoremstyle{remark}
}
{\theoremstyle{remark}
}

\numberwithin{equation}{section}

\title{Modules over \'etale groupoid algebras as sheaves}

\author{Benjamin Steinberg}
\address{%
    Department of Mathematics\\
    City College of New York\\
    Convent Avenue at 138th Street\\
    New York, New York 10031\\
    USA}
\email{bsteinberg@ccny.cuny.edu}

\thanks{This work was partially supported by a grant from the Simons Foundation(\#245268
to Benjamin Steinberg) and the Binational Science Foundation of Israel and the US (\#2012080 to Benjamin Steinberg).}
\date{May 31, 2014}

\keywords{\'etale groupoids, groupoid algebras, sheaves}
\subjclass[2010]{16S99,16S10, 22A22, 18F20}

\begin{document}

\begin{abstract}
The author has previously associated to each commutative ring with unit $\Bbbk$ and \'etale groupoid $\mathscr G$ with locally compact, Hausdorff, totally disconnected unit space a $\Bbbk$-algebra $\Bbbk\mathscr G$.  The algebra $\Bbbk\mathscr G$ need not be unital, but it always has local units.  The class of groupoid algebras includes group algebras, inverse semigroup algebras and Leavitt path algebras.  In this paper we show that the category of unitary $\Bbbk\mathscr G$-modules is equivalent to the category of sheaves of $\Bbbk$-modules over $\mathscr G$. As a consequence we obtain a new proof of a recent result that Morita equivalent groupoids have Morita equivalent algebras.
\end{abstract}

\maketitle

\section{Introduction}
In an effort to obtain a uniform theory for group algebras, inverse semigroup algebras and Leavitt path algebras~\cite{Leavitt}, the author~\cite{discretegroupoid} associated to each commutative ring with unit $\Bbbk$ and \'etale groupoid $\mathscr G$ with locally compact, totally disconnected unit space a $\Bbbk$-algebra $\Bbbk\mathscr G$ (see also~\cite{theirgroupoid}). These algebras are discrete versions of groupoid $C^*$-algebras~\cite{renault,Paterson} and a number of analogues of results from the operator theoretic setting have been obtained in this context. In particular, Cuntz-Krieger uniqueness theorems~\cite{simplicity1,simplicity2}, characterizations of simplicity~\cite{simplicity1,simplicity2} and the connection of groupoid Morita equivalence to Morita equivalence of algebras~\cite{GroupoidMorita} have  been proven for groupoid algebras under the Hausdorff assumption.

In this paper, we prove a discrete analogue of Renault's disintegration theorem~\cite{renaultdisintegration}, which roughly states that representations of groupoid $C^*$-algebras are obtained by integrating representations of the groupoid. A representation of a groupoid consists of a field of Hilbert spaces over the unit space with an action of the groupoid by unitary transformations on the fibers~\cite{renault,Paterson}.  

Here we prove that the category of unitary $\Bbbk\mathscr G$-modules is equivalent to the category of sheaves of $\Bbbk$-modules over $\mathscr G$.  This simultaneously generalizes the following two well-known facts: if $X$ is a locally compact, totally disconnected space and $C_c(X,\Bbbk)$ is the ring of locally constant functions $X\to \Bbbk$ with compact support, then the category of unitary $C_c(X,\Bbbk)$-modules is equivalent to the category of sheaves on $X$ (cf.~\cite{Piercespectrum}); and if $\mathscr G$ is a discrete groupoid, then the category of unitary $\Bbbk\mathscr G$-modules is equivalent to the category of functors from $\mathscr G$ to the category of $\Bbbk$-modules (cf.~\cite{ringoids}).  In the context of \'etale Lie groupoids and convolution algebras of smooth functions, analogous results can be found in~\cite{liecase}. However, the techniques in the totally disconnected case setting are quite different.

As a consequence of our results, we obtain a new proof that Morita equivalent groupoids have Morita equivalent groupoid algebras, which the author feels is more conceptual than the one in~\cite{GroupoidMorita} (since it works  with module categories rather than Morita contexts), and at the same time does not require the Hausdorff hypothesis.  

We hope that this geometric realization of the module category will prove useful in the study of simple modules, primitive ideals and in other contexts analogous to those in which Renault's disintegration theorem is used in operator theory.

\section{\'Etale groupoids}
In this paper, a topological space will be called compact if it is Hausdorff and satisfies the property that every open cover has a finite subcover.

\subsection{Topological groupoids}
A topological groupoid $\mathscr G$ is a groupoid (i.e., a small category each of whose morphisms is an isomorphism) whose unit space $\mathscr G\skel 0$ and arrow space $\mathscr G\skel 1$ are topological spaces and whose domain map $\dom$, range map $\ran$, multiplication map, inversion map and unit map $u\colon \mathscr G\skel 0\to \mathscr G\skel 1$ are all continuous.  Since $u$ is a homeomorphism with its image, we identify elements of $\mathscr G\skel 0$ with the corresponding identity arrows and view $\mathscr G\skel 0$ as a subspace of $\mathscr G\skel 1$ with the subspace topology.  We write $\mathscr G\skel 2$ for the space of composable arrows $(g,h)$ with $\dom(g)=\ran(h)$.

A topological groupoid $\mathscr G$ is \emph{\'etale} if $\dom$ is a local homeomorphism.  This implies that $\ran$ and the multiplication map are local homeomorphisms  and that $\mathscr G\skel 0$ is open in $\mathscr G\skel 1$~\cite{resendeetale}.  Note that the fibers of $\dom$ and $\ran$ are discrete in the induced topology.  A \emph{local bisection} of $\mathscr G$ is an open subset $U\subseteq \mathscr G\skel 1$ such that $\dom|_U$ and $\ran|_U$ are homeomorphisms to their images.  The set of local bisections of $\mathscr G$, denoted $\Bis(\mathscr G)$, is a basis for the topology on $\mathscr G\skel 1$~\cite{resendeetale,Paterson,Exel}.  If $U,V$ are local bisections, then
\begin{align*}
UV&=\{uv\mid u\in U,v\in V\}\\
U\inv &= \{u\inv\mid u\in U\}
\end{align*}
are local bisections. In fact, $\Bis(\mathscr G)$ is an inverse semigroup~\cite{Lawson}.

An \'etale groupoid is said to be \emph{ample}\cite{Paterson} if $\mathscr G\skel 0$ is Hausdorff and has a basis of compact open sets.  In this case $\mathscr G\skel 1$ is locally Hausdorff but need not be Hausdorff.  Let $\Bis_c(\mathscr G)$ denote the set of compact open local bisections of $\mathscr G$.  Then $\Bis_c(\mathscr G)$ is an inverse subsemigroup of $\Bis(\mathscr G)$ and is a basis for the topology of $\mathscr G\skel 1$~\cite{Paterson}.

\subsection{$\mathscr G$-sheaves and Morita equivalence}
Let $\mathscr G$ be an \'etale groupoid. References for this section are~\cite{classifyingtoposI,classifyingtoposII,mrcun,Moerdijkgroupoid,groupoidhomology,elephant1, elephant2}.  A \emph{(right) $\mathscr G$-space} consists of a space $E$, a continuous map $p\colon E\to \mathscr G\skel 0$ and an action map $E\times_{\mathscr G\skel 0} \mathscr G\skel 1\to E$ (where the fiber product is with respect to $p$ and $\ran$), denoted $(x,g)\mapsto xg$ satisfying the following axioms:
\begin{itemize}
\item $ep(e)=e$ for all $e\in E$;
\item $p(eg)=\dom(g)$ whenever $p(e)=\ran(g)$;
\item $(eg)h=e(gh)$ whenever $p(e)=\ran(g)$ and $\dom(g)=\ran(h)$.
\end{itemize} 
Left $\mathscr G$-spaces are defined dually.

A $\mathscr G$-space $(E,p)$ is said to be
\emph{principal} if the natural map \[E\times_{\mathscr G\skel 0} \mathscr G\skel 1\to E\times_{\mathscr G\skel 0} E\] given by $(e,g)\mapsto (eg,e)$ is a homeomorphism. A \emph{morphism} $(E,p)\to (F,q)$ of $\mathscr G$-spaces is a continuous map $\varphi\colon E\to F$ such that
\[\xymatrix{E\ar[rd]_p\ar[rr]^\varphi &  &F\ar[ld]^q\\ & \mathscr G\skel 0 &}\] commutes and $\varphi(eg)=\varphi(e)g$ whenever $p(e)=\ran(g)$.

Morita equivalence plays an important role in groupoid theory.  There are a number of different, but equivalent, formulations of the notion.  See~\cite{classifyingtoposI,classifyingtoposII,mrcun,Moerdijkgroupoid,elephant1, elephant2,groupoidhomology} for details.
Two topological groupoids $\mathscr G$ and $\mathscr H$ are said to be \emph{Morita equivalent} if there is a topological space $E$ with the structure of a principal left $\mathscr G$-space $(E,p)$ and a principal right $\mathscr H$-space $(E,q)$ such that $p,q$ are open surjections and the actions commute, meaning $p(eh)=p(e)$, $q(ge)=q(e)$ and $(ge)h=g(eh)$ whenever $g\in \mathscr G\skel 1$, $h\in \mathscr H\skel 1$ and $\dom(g)=p(e)$, $\ran(h)=q(e)$.

A continuous functor $f\colon \mathscr G\to \mathscr H$ of \'etale groupoids is called an \emph{essential equivalence} if $\dom\pi_2\colon \mathscr G\skel 0\times_{\mathscr H\skel 0}\times \mathscr H\skel 1\to \mathscr H\skel 0$ is an open surjection (where the fiber product is over $f$ and $\ran$) and the square
\[\xymatrix{\mathscr G\skel 1\ar[rr]^f\ar[d]_{(\dom,\ran)} && \mathscr H\skel 1\ar[d]^{(\dom,\ran)}\\ \mathscr G\skel 0\times \mathscr G\skel 0\ar[rr]_{f\times f} && \mathscr H\skel 0\times \mathscr H\skel 0} \] is a pullback.  The first condition corresponds to being essentially surjective and the second to being fully faithful in the discrete context. \'Etale groupoids $\mathscr G$ and $\mathscr H$ are Morita equivalent if and only if there is an \'etale groupoid $\mathscr K$ and essential equivalences $f\colon \mathscr K\to \mathscr G$ and $f'\colon \mathscr K\to \mathscr H$.

If $\mathscr G$ is an \'etale groupoid, then a \emph{$\mathscr G$-sheaf} consists of a $\mathscr G$-space $(E,p)$ such that $p\colon E\to \mathscr G\skel 0$ is a local homeomorphism (the tradition to use right actions is standard in topos theory).  The fiber $p\inv(x)$ of $E$ over $x$ is denoted $E_x$ and is called the \emph{stalk} of $E$ at $x$.  A morphism of $\mathscr G$-sheaves is just a morphism of $\mathscr G$-spaces; note, however, that the corresponding map of total spaces must necessarily be a local homeomorphism. The category $\mathcal B\mathscr G$ of all $\mathscr G$-sheaves is called the \emph{classifying topos} of $\mathscr G$~\cite{elephant1, elephant2}.

If $A$ is a set, then the constant $\mathscr G$-sheaf $\Delta(A)$ is $(A\times \mathscr G\skel 0,\pi_2)$ with action $(a,\ran(g))g=(a,\dom(g))$ (where $A$ is endowed with the discrete topology). As a sheaf over $\mathscr G\skel 0$, note that $\Delta(A)$ is nothing more than the sheaf of locally constant $A$-valued functions on $\mathscr G\skel 0$. (Recall that a locally constant function from a topological space $X$ to a set $A$ is just a continuous map $X\to A$ where $A$ is endowed with the discrete topology.) The functor $\Delta\colon \mathbf{Set}\to \mathcal B\mathscr G$ is exact and hence sends rings to internal rings of $\mathcal B\mathscr G$.  If $\Bbbk$ is a commutative ring with unit, then a \emph{$\mathscr G$-sheaf of $\Bbbk$-modules} is by definition an internal $\Delta(\Bbbk)$-module in $\mathcal B\mathscr G$. Explicitly, this amounts to a $\mathscr G$-sheaf $(E,p)$ together a $\Bbbk$-module structure on each stalk $E_x$ such that:
\begin{itemize}
\item the zero section, denoted $0$, sending $x\in \mathscr G\skel 0$ to the zero of $E_x$ is continuous;
\item addition $E\times_{\mathscr G\skel 0} E\to E$ is continuous;
\item scalar multiplication $K\times E\to E$ is continuous;
\item for each $g\in \mathscr G\skel 1$, the map $R_g\colon E_{\ran(g)}\to E_{\dom(x)}$ given by $R_g(e) = eg$ is $\Bbbk$-linear;
\end{itemize}
where $\Bbbk$ has the discrete topology in the third item.
Note that the first three conditions are equivalent to $(E,p)$ being a sheaf of $\Bbbk$-modules over $\mathscr G\skel 0$.

Internal $\Delta(\Bbbk)$-module homomorphisms are just $\mathscr G$-sheaf morphisms
 \[\xymatrix{E\ar[rd]_p\ar[rr]^\varphi &  &F\ar[ld]^q\\ & \mathscr G\skel 0 &}\]
 which restrict to $\Bbbk$-module homomorphisms on the stalks. The category of $\mathscr G$-sheaves of $\Bbbk$-modules will be denoted $\mathcal B_{\Bbbk}\mathscr G$.

It follows from standard topos theory that if $\mathcal B\mathscr G$ is equivalent to $\mathcal B\mathscr H$, then the equivalence commutes with the constant functor (up to isomorphism) and hence yields an equivalence of categories  $\mathcal B_{\Bbbk}\mathscr G$ and  $\mathcal B_{\Bbbk}\mathscr H$. Indeed, $\Delta$ is left adjoint to the hom functor out of the terminal object and equivalences preserve terminal objects.

Moerdijk proved that if $\mathscr G$ and $\mathscr H$ are \'etale groupoids, then $\mathcal B\mathscr G$ is equivalent to $\mathcal B\mathscr H$ if and only if $\mathscr G$ and $\mathscr H$ are Morita equivalent groupoids~\cite{classifyingtoposI,classifyingtoposII,Moerdijkgroupoid,groupoidhomology,elephant1, elephant2}.  Hence Morita equivalent groupoids have equivalent categories of sheaves of $\Bbbk$-modules for any base ring $\Bbbk$.

\subsection{Groupoid algebras}
Let $\mathscr G$ be an ample groupoid and $\Bbbk$ a commutative ring with unit.  Define $\Bbbk\mathscr G$ to be the $\Bbbk$-submodule of $\Bbbk^{\mathscr G\skel 1}$ spanned by the characteristic functions $\chi_U$ with $U\in \Bis_c(\mathscr G)$.  If $\mathscr G\skel 1$ is Hausdorff, then $\Bbbk \mathscr G$ consists precisely of the locally constant functions $\mathscr G\skel 1\to \Bbbk$ with compact support; otherwise, it is the $\Bbbk$-submodule spanned by those functions $f\colon \mathscr G\skel 1\to \Bbbk$ that vanish outside some Hausdorff open subset $U$ with $f|_U$ locally constant with compact support.  See~\cite{discretegroupoid,groupoidarxiv,theirgroupoid} for details.

The convolution product on $\Bbbk \mathscr G$, defined by
\[f_1\ast f_2(g)=\sum_{\dom(h)=\dom(g)} = f_1(gh\inv)f_2(h), \] turns $\Bbbk \mathscr G$ into a $\Bbbk$-algebra.  Note that the sum is finite because the fibers of $\dom$ are closed and discrete, and $f_1,f_2$ are linear combinations of functions with compact support. We often just write for convenience $f_1f_2$ instead of $f_1\ast f_2$.  One has that $\chi_U\chi_V=\chi_{UV}$ for $U,V\in \Bis_c(\mathscr G)$; see~\cite{discretegroupoid}.  If $\mathscr G\skel 1=\mathscr G\skel 0$, then the convolution product is just pointwise multiplication and so $\Bbbk \mathscr G$ is just the usual ring of locally constant functions $\mathscr G\skel 0\to \Bbbk$ with compact support.

The ring $\Bbbk \mathscr G$ is unital if and only if $\mathscr G\skel 0$ is compact~\cite{discretegroupoid}. However, it is very close to being unital in the following sense. A ring $R$ is said to have \emph{local units} if it is a direct limit of unital rings in the category of not necessarily unital rings (that is, the homomorphisms in the directed system do not have to preserve the identities).  Equivalently, $R$ has local units if, for any finite subset $r_1,\ldots, r_n$ of $R$, there is an idempotent $e\in R$ with $er_i=r_i=r_ie$ for $i=1,\ldots, n$~\cite{Abrams,Laci}. Denote by $E(R)$ the set of idempotents of $R$.

\begin{Prop}\label{p:directedunion}
Let $\mathscr G$ be an ample groupoid and $\Bbbk$ a commutative ring with units.  Let $\mathcal B$ denote the generalized boolean algebra of compact open subsets of $\mathscr G\skel 0$. If $U\in \mathcal B$, let $\mathscr G_U=(U,\dom\inv(U)\cap \ran\inv(U))$.
\begin{enumerate}
\item $\mathcal B$ is directed.
\item If $U\in \mathcal B$, then $\mathscr G_U$ is an open ample subgroupoid of $\mathscr G$.
\item If $U\in \mathcal B$, then $\chi_U\cdot \Bbbk \mathscr G\cdot \chi_U\cong \Bbbk \mathscr G_U$.
\item $\Bbbk \mathscr G=\bigcup_{U\in \mathcal B}\chi_U\cdot\Bbbk \mathscr G\cdot\chi_U = \varinjlim_{U\in \mathcal B} \Bbbk \mathscr G_U$.
\end{enumerate}
In particular, $\Bbbk\mathscr G$ has local units.
\end{Prop}
\begin{proof}
Clearly, $\mathcal B$ is directed since the union of two elements is their join.  Also $\mathscr G_U$ is an open ample subgroupoid of $\mathscr G$.  It follows that $\Bbbk \mathscr G_U$ can be identified with a subalgebra of $\Bbbk \mathscr G$ by extending functions on $\mathscr G_U\skel 1$ to be $0$ outside of $\mathscr G_U\skel 1$.  Since $\chi_U$ is the identity of $\Bbbk\mathscr G_U$ (cf.~\cite{discretegroupoid}), $\Bbbk \mathscr G_U$ is a unital subring of $\chi_U\cdot\Bbbk \mathscr G\cdot\chi_U$.  But if $f\in \Bbbk \mathscr G$ and $g\notin \mathscr G_U\skel 1$, then $(\chi_U\cdot f\cdot \chi_U)(g) =0$.  Thus $\chi_U\cdot\Bbbk \mathscr G\cdot\chi_U=\Bbbk\mathscr G_U$.

Let $R=\bigcup_{U\in \mathcal B}\chi_U\cdot\Bbbk\mathscr G_U\cdot\chi_U$.  Then $R$ is a $\Bbbk$-subalgebra of $\Bbbk \mathscr G$ because $\mathcal B$ is directed.  To show that $R$ is the whole ring, we just need to show it contains the spanning set $\chi_U$ with $U\in \Bis_c(\mathscr G)$.  Put $V=U\inv U\cup UU\inv$.  Then $V\in \mathcal B$ and $\chi_V\cdot\chi_U\cdot \chi_V=\chi_{VUV}=\chi_U$.
\end{proof}

Examples of groupoid algebras of ample groupoids include group algebras, Leavitt path algebras~\cite{theirgroupoid,GroupoidMorita} and inverse semigroup algebras~\cite{discretegroupoid}, as well as discrete groupoid algebras and certain cross product and partial action cross product algebras.  In general, groupoid algebras allow one to construct discrete analogues of a number of classical $C^*$-algebras that can be realized as $C^*$-algebras of ample groupoids~\cite{renault,Paterson,Exel}.

\section{The equivalence theorem}
Fix an ample groupoid $\mathscr G$ and a commutative ring with unit $\Bbbk$.  Our goal is to establish an equivalence between the category $\modu\Bbbk \mathscr G$ of unitary right $\Bbbk\mathscr G$-modules and the category $\mathcal B_\Bbbk\mathscr G$ of $\mathscr G$-sheaves of $\Bbbk$-modules. Let us recall the missing definitions.

If $R$ is a ring with local units, a right $R$-module $M$ is \emph{unitary} if $MR=M$, or  equivalently, for each $m\in M$, there is an idempotent $e\in E(R)$ with $me=m$.  We write $\modu R$ for the category of unitary right $R$-modules.  Two rings $R,S$ with local units are \emph{Morita equivalent} if $\modu R$ is equivalent to $\modu S$~\cite{Abrams,Laci,Simon}.  One can equivalently define Morita equivalence in terms of unitary left modules and in terms of Morita contexts~\cite{Abrams,Laci,Simon}.

Suppose that $R$ is a $\Bbbk$-algebra with local units.  Then we note that every unitary $R$-module is a $\Bbbk$-module and the $\Bbbk$-module structure is compatible with the $\Bbbk$-algebra structure.  Indeed, if $e\in E(R)$, then $Me$ is a unital $eMe$-module and hence a $\Bbbk$-module in the usual way.  As $M$ is unitary, it follows that $M$ is the directed union $\bigcup_{e\in E(M)}Me$ and hence a $\Bbbk$-module. More concretely, the $\Bbbk$-module structure is given as follows: if $c\in \Bbbk$ and $m\in M$, then $cm=m(ce)$ where $e$ is any idempotent such that $me=m$.    The $\Bbbk$-module structure is then automatically preserved by any $R$-module homomorphism, as in the case of unital rings.

Define a functor $\Gamma_c\colon \mathcal B_\Bbbk\mathscr G\to \modu\Bbbk\mathscr G$ as follows.  If $(E,p)$ is a $\mathscr G$-sheaf of $\Bbbk$-modules, then $\Gamma_c(E,p)$ is the set of all compactly supported (global) sections $s\colon \mathscr G\skel 0\to E$ of $p$ with pointwise addition.   We define a $\Bbbk \mathscr G$-module structure by \[(sf)(x) = \sum_{\dom(g)=x} f(g)s(\ran(g))g=\sum_{\dom(g)=x} f(g)R_g(s(\ran(g))). \]  As usual, the sum is finite because $f$ is a finite sum of functions with compact support and the fibers of $\dom$ are closed and discrete. It is easy to check that this makes $\Gamma_c(E,p)$ into a $\Bbbk\mathscr G$-module and that the induced $\Bbbk$-module structure is just the pointwise one.  The following observation is so fundamental that we shall often use it without comment throughout.

\begin{Prop}\label{p:bisecaction}
If $U\in \Bis_c(\mathscr G)$ and $s\in \Gamma_c(E,p)$, then
\[(s\chi_U)(x)=\begin{cases} s(\ran(g))g, & \text{if}\ g\in U,\ \dom(g)=x\\ 0,  & \text{if}\ x\notin U\inv U.\end{cases}\]  In particular, if $U\subseteq \mathscr G\skel 0$ is compact open, then $(s\chi_U)=\chi_U(x)s(x)$.
\end{Prop}

The module $\Gamma_c(E,p)$ is unitary because if $s\colon \mathscr G\skel 0\to E$ has compact support, then we can find a compact open set $U$ containing the support of $s$ (just cover the support by compact open sets and take the union of a finite subcover).  Then one readily checks that $s\chi_U = s$ using Proposition~\ref{p:bisecaction}.

If \[\xymatrix{E\ar[rd]_p\ar[rr]^\varphi && F\ar[ld]^q\\ & \mathscr G\skel 0 &}\] is a morphism of $\mathscr G$-sheaves of $\Bbbk$-modules and $s\in \Gamma_c(E,p)$, then define $\Gamma_c(\varphi)(s) = \varphi\circ s$. It is straightforward to verify that $\Gamma_c$ is a functor.

Conversely, let $M$ be a unitary right $\Bbbk\mathscr G$-module.  We define a $\mathscr G$-sheaf $\Sh(M)=(\til M,p_M)$ in steps.  Recall that we have been using $\mathcal B$ to denote the generalized boolean algebra of compact open subsets of $\mathscr G\skel 0$. For each $U\in \mathcal B$, we can consider the $\Bbbk$-submodule $M(U)=M\chi_U$.  If $U\subseteq V$, then $M(U)=M\chi_U=M\chi_{UV}=M\chi_U\chi_V\subseteq M\chi_V=M(V)$.  Note that $M(U)$ is a $\Bbbk \mathscr G_U=\chi_U\cdot\mathscr G\cdot\chi_U$-module and since $\Bbbk \mathscr G=\bigcup_{U\in \mathcal B}\chi_U\cdot\mathscr G\cdot\chi_U=\varinjlim_{U\in \mathcal B}\Bbbk \mathscr G_U$, it follows that $M=\bigcup_{U\in \mathcal B}M(U) = \varinjlim_{U\in \mathcal B} M(U)$, where the latter has the obvious module structure coming from $\Bbbk \mathscr G=\varinjlim_{U\in \mathcal B}\Bbbk G_U$.

Let $x\in \mathscr G\skel 0$. If $x\in V\subseteq U$ with $U,V\in \mathcal B$, then we have a $\Bbbk$-module homomorphism $\rho^U_V\colon M(U)\to M(V)$ given by $m\mapsto m\chi_V$.  Since $\rho^U_U=1_{M(U)}$ and if $W\subseteq V\subseteq U$, we have $\rho^V_W\circ \rho^U_V=\rho^U_W$, it follows that we can form the direct limit $\Bbbk$-module $M_x=\varinjlim_{x\in U} M(U)$.  If $m\in M(U)$, we let $[m]_x$ denote the equivalence class of $m$ in $M_x$. Since $M=\bigcup_{U\in \mathcal B} M(U)$ and each element of $\mathcal B$ is contained in an element which contains $x$, it follows that $[m]_x$ is defined for all $m\in M$ and $m\mapsto [m]_x$ gives a $\Bbbk$-linear map $M\to M_x$.

Put $\til M=\coprod_{x\in \mathscr G\skel 0} M_x$ and let $p_M(M_x)=x$ for $x\in \mathscr G\skel 0$.  Let $U$ be a compact open subset of $\mathscr G\skel 0$ and let $m\in M$.  Define \[(U,m)=\{[m]_x\mid x\in U\}\subseteq \til M.\] Suppose that $[m]_x\in (U,m_1)\cap (V,m_2)$.  Then there is a compact open neighborhood  $W\subseteq U\cap V$ of $x$ such that $m\chi_W=m_1\chi_W=m_2\chi_W$.  It follows that $[m]_x\in (W,m)\subseteq (U,m_1)\cap (V,m_2)$ and hence the sets $(U,m)$ form a basis for a topology on $\til M$.  Continuity of $p_M$ follows because if $U$ is a compact open subset of $\mathscr G\skel 0$, then $p_M\inv(U)=\bigcup_{m\in M} (U,m)$ is open.  Trivially, $p_M$ takes $(U,m)$ bijectively to $U$ and is thus a local homeomorphism.

Each stalk $M_x$ is a $\Bbbk$-module.  We must show continuity of the $\Bbbk$-module structure.  To establish continuity of the zero section $x\mapsto [0]_x$, suppose that $(U,m)$ is a basic neighborhood of $[0]_x$.  Then there is a compact open neighborhood $W$ of $x$ with $W\subseteq U$ and $m\chi_W=0$.  Then, for all $z\in W$, one has $[m]_z=[0]_z$ and so the zero section maps $W$ into $(U,m)$.  Thus the zero section is continuous.

To see that scalar multiplication is continuous, let $k\in \Bbbk$ and suppose that $[kn]_x=k[n]_x\in (U,m)$.  Then there is a compact open neighborhood $W$ of $x$ with $W\subseteq U$ and $kn\chi_W=m\chi_W$.  If $(k,[n]_z)\in \{k\}\times (W,n)$, then $k[n]_z=[kn]_z=[m]_z$ because $z\in W$ and $kn\chi_W=m\chi_W$.   This yields continuity of scalar multiplication.

Continuity of addition is proved as follows.  Suppose $(U,m)$ is a basic neighborhood of $[m_1]_x+[m_2]_x=[m_1+m_2]_x$.  Then there is a compact open neighborhood $W$ of $x$ with $W\subseteq U$ and $(m_1+m_2)\chi_W=m\chi_W$.  Therefore, if $([m_1]_z,[m_2]_z)\in \left((m_1,W)\times (m_2,W)\right)\cap (\til M\times_{\mathscr G\skel 0}\til M)$, then $[m_1]_z+[m_2]_z = [m_1+m_2]_z =[m]_z\in (U,m)$. Therefore, addition is continuous.

Next, we must define the $\mathscr G$-action.  Define, for $g\in \mathscr G\skel 1$, a mapping $R_g\colon M_{\ran(g)}\to M_{\dom(g)}$ by $R_g([m]_{\ran(g)}) = [m\chi_U]_{\dom(x)}$ where $U$ is a compact local bisection containing $g$.  We also write $R_g([m]_{\ran(g)}) = [m]_{\ran(g)}g$.

\begin{Prop}\label{p:actionprops}
The following hold.
\begin{enumerate}
\item $R_g$ is a well-defined $\Bbbk$-module homomorphism.
\item If $(g,h)\in \mathscr G\skel 2$, then $([m]_{\ran(g)}g)h = [m]_{\ran(gh)}(gh)$.
\item If $x\in \mathscr G\skel 0$, then $[m]_xx=[m]_x$.
\end{enumerate}
\end{Prop}
\begin{proof}
Suppose that $g\colon y\to x$.  To show that $R_g$ is well defined, let $[m]_x=[n]_x$ and let $U,V\in \Bis_c(\mathscr G)$ with $g\in U\cap V$.  Then there exist a compact open neighborhood $W$ of $x$ with $m\chi_W=n\chi_W$ and $Z\in \Bis_c(\mathscr G)$ such that $g\in Z\subseteq U\cap V$.  Note that $g\in WZ\subseteq U\cap V$ and so $y\in Z\inv WZ\subseteq \mathscr G\skel 0$.  Also we compute
\begin{align*}
m\chi_U\chi_{Z\inv WZ} &= m\chi_{UZ\inv WZ}=m\chi_{U(WZ)\inv WZ}=m\chi_{WZ}=m\chi_W\chi_Z\\ &=n\chi_W\chi_Z
=n\chi_{WZ}=n\chi_{V(WZ)\inv WZ}= n\chi_{VZ\inv WZ}\\ &=n\chi_V\chi_{Z\inv WZ}
\end{align*}
which shows that $[m\chi_U]_y=[n\chi_V]_y$, i.e., $R_g$ is well defined.  Clearly $R_g$ is $\Bbbk$-linear.

Suppose now that $(g,h)\in \mathscr G\skel 2$.  Choose $U,V\in \Bis_c(\mathscr G)$ such that $g\in U$ and $h\in V$. Then $gh\in UV$ and so if $g\colon y\to x$ and $h\colon z\to y$, then \[([m]_xg)h = [m\chi_U]_yh=[m\chi_U\chi_V]_z=[m\chi_{UV}]_z=[m]_z(gh)\] as required.

Finally, if $x\in \mathscr G\skel 0$ and $U$ is a compact open neighborhood of $x$ in $\mathscr G\skel 0$, then $[m]_xx=[m\chi_U]_x=[m]_x$ by definition of $M_x$.
\end{proof}

In light of Proposition~\ref{p:actionprops}, there is an action map $\til M\times_{\mathscr G\skel 0} \mathscr G\skel 1\to \til M$ given by $([m]_{\ran(g)},g)\mapsto [m]_{\ran(g)}g$ satisfying $p_M([m]_{\ran(g)}g) = \dom(g)$. To prove that $\Sh(M)$ is a $\mathscr G$-sheaf of $\Bbbk$-modules, it remains to show the action map is continuous.  Let $g\colon y\to x$ and let $(m,U)$ be a basic neighborhood of $[n]_xg$.  Then $y\in U$ and $[m]_y=[n]_xg$. Let $V\in \Bis_c(\mathscr G)$ with $g\in V$.  Then $[n\chi_V]_y =[m]_y$ and so there exists a compact open neighborhood $W\subseteq U$ of $y$ with $n\chi_{VW}=n\chi_V\chi_W=m\chi_W$.  Note that $g\in VW$ and $x\in VWV\inv\subseteq \mathscr G\skel 0$.  Consider the neighborhood $N=\left((n,VWV\inv)\times VW\right)\cap \left(\til M\times_{\mathscr G\skel 0} \mathscr G\skel 1\right)$ of $([n]_x,g)$.  If $([n]_z,h)\in N$, with $h\colon z'\to z$, then because $h\in VW$, we have $[n]_zh=[n\chi_{VW}]_{z'} = [m\chi_W]_{z'}=[m]_{z'}\in (m,U)$ as $z'\in V\inv VW\subseteq W\subseteq U$.  This establishes that $(\til M,p)$ is a $\mathscr G$-sheaf of $\Bbbk$-modules.

Next suppose that $f\colon M\to N$ is a $\Bbbk\mathscr G$-module homomorphism. The $f(M(U)) = f(M\chi_U) = f(M)\chi_U\subseteq N\chi_U=N(U)$.  Thus there is an induced $\Bbbk$-linear map $f_x\colon M_x\to N_x$ given by $f([m]_x) = [f(m)]_x$.  Define
\[\xymatrix{\til M\ar[rd]_{p_M}\ar[rr]^{\Sh(f)} && \til N\ar[ld]^{p_N}\\ & \mathscr G\skel 0 &}\] by $\Sh(f)([m]_x)=f_x([m]_x)$.
First we check that $\Sh(f)$ preserves the action.  Suppose $g\colon y\to x$ and $U\in \Bis_c(\mathscr G)$ with $g\in U$.  Then \[f_y([m]_xg) = f_y([m\chi_U]_y) = [f(m\chi_U)]_y=[f(m)\chi_U]_y=[f(m)]_xg=f_x([m]_x)g\] as required.

It remains to check continuity of $\Sh(f)$.  Let $[m]_x\in \til M$ and let $(U,n)$ be a basic neigborhood of $f_x([m]_x)$.  Then $x\in U$ and $f_x([m]_x)=[f(m)]_x=[n]_x$.  Choose a compact open neighborhood $W$ of $x$ contained in $U$ such that $f(m)\chi_W=n\chi_W$.  Consider the neighborhood $(W,m)$ of $[m]_x$.  If $[m]_z\in (W,m)$, then $f_z([m]_z) = [f(m)]_z= [f(m)\chi_W]_z=[n\chi_W]_z=[n]_z$ because $z\in W$.  Thus $\Sh(f)(W,m)\subseteq (U,n)$, yielding the continuity of $\Sh(f)$.  It is obvious that $\Sh$ is a functor.

The following lemma will be useful for proving that these functors are quasi-inverse.

\begin{Lemma}\label{l:useful}
Let $M\in \modu \Bbbk \mathscr G$.  If $U\in \Bis_c(\mathscr G)$ and $x\notin U\inv U$, then $[m\chi_U]_x=0$.
\end{Lemma}
\begin{proof}
Since $U\inv U$ is compact and $\mathscr G\skel 0$ is Hausdorff, we can find a compact open neighborhood $W$ of $x$ with $W\cap U\inv U=\emptyset$.  Then $m\chi_U\chi_W=m\chi_{UW}=0$.
\end{proof}

\begin{Thm}\label{t:equivalence}
There are natural isomorphisms $\Gamma_c\circ \Sh\cong 1_{\modu \Bbbk G}$ and $\Sh\circ \Gamma_c\cong 1_{\mathcal B_\Bbbk\mathscr G}$.  Hence the categories $\modu \Bbbk \mathscr G$ and $\mathcal B_\Bbbk\mathscr G$ are equivalent.
\end{Thm}
\begin{proof}
Let $M$ be a unitary $\Bbbk \mathscr G$-module and define $\eta_M\colon M\to \Gamma_c(\Sh(M))$ by $\eta_M(m) = s_m$ where $s_m(x) = [m]_x$ for all $x\in X$.  We claim that $s_m$ is continuous with compact support.  Continuity is easy:  if $s_m(x)\in (U,n)$, then $x\in U$ and $[m]_x=[n]_x$.  So there is a compact open neighborhood $W$ of $x$ with $W\subseteq U$ and $m\chi_W=n\chi_W$.  Then if $z\in W$, we have $s_m(z)=[m]_z=[m\chi_W]_z=[n\chi_W]_z=[n]_z\in (U,n)$.  Thus $s_m$ is continuous.    We claim that the support of $s_m$ is compact.  Let $U\in \mathcal B$ with $m\chi_U=m$.  Suppose that $x\notin U$. Then Lemma~\ref{l:useful} implies that $[m]_x=[m\chi_U]_x=0$. Thus the support of $s_m$ is a closed subset of $U$ and hence compact.

We claim that $\eta_M$ is an isomorphism (it is clearly natural in $M$).  Let us first show that $\eta_M$ is a module homomorphism.  It is clearly $\Bbbk$-linear and hence it suffices to show that if $U\in \Bis_c(\mathscr G)$, then $\eta_M(m\chi_U) = \eta_M(m)\chi_U$. Note that $\eta_M(m\chi_U) = s_{m\chi_U}$.   If $x\notin U\inv U$, then $s_{m\chi_U}(x) = [m\chi_U]_x=0$ by Lemma~\ref{l:useful}. If $x=\dom(g)$ with $g\in U$, then we have $s_{m\chi_U}(x) = [m\chi_U]_x=[m]_{\ran(g)}g=s_m(\ran(g))g$.  Therefore, in light of Proposition~\ref{p:bisecaction}, we conclude that $s_{m\chi_U}=s_m\chi_U$.
%On the other hand, if $x\notin U\inv U$, then $(s_m\chi_U)(x) = 0$, whereas if $x\in U\inv U$, then $(s_m\chi_U)(x)= [m]_{\ran(g)}g$.
This shows that $\eta_M$ is a $\Bbbk \mathscr G$-module homomorphism.

Suppose that $0\neq m\in M$. Then $m\chi_U=m$ for some $U\in \mathcal B$.  Let $\mathcal B_U$ be the boolean ring of compact open subsets of $U$.  Let $I$ be the ideal of $\mathcal B_U$ consisting of those $V$ with $m\chi_V=0$.   This is a proper ideal (since $U\notin I$) and hence contained in a maximal ideal $\mathfrak m$. Let $x$ be the point of $U$ corresponding to $\mathfrak m$ under Stone duality.  Then the ultrafilter of compact open neighborhoods of $x$ is $\mathcal B\setminus \mathfrak m$ and hence does not intersect $I$. Therefore, $m\chi_V\neq 0$ for all compact open neighborhoods of $x$, that is, $s_m(x)=[m]_x\neq 0$. Therefore, $\eta_M(m)=s_m\neq 0$ and so $\eta_M$ is injective.

To see that $\eta_M$ is surjective, let $s\in \Gamma_c(\Sh(M))$ and let $K$ be the support of $s$. For each $x\in K$, we can find a compact open neighborhood $U_x$ of $x$ and an element $m_x\in M$ such that $s(z)=[m_x]_z$ for all $z\in U_x$ (choose $U_x$ mapping under $s$ into a basic neighborhood of $s(x)$ of the form $(V_x,m_x)$).  By compactness of $K$, we can find a finite subcover of the $U_x$ with $x\in X$.  Since $\mathscr G\skel 0$ is Hausdorff, we can refine the subcover by a partition into compact open subsets, that is, we can find disjoint compact open sets $V_1,\ldots, V_n$ and elements $m_1,\ldots, m_n\in M$ such that $K\subseteq V_1\cup\cdots\cup V_n$ and $s(x)=[m_i]_x$ for all $x\in V_i$.  Consider $m=m_1\chi_{V_1}+\cdots+m_n\chi_{V_n}$.  Then $m\chi_{V_i} = m_i\chi_{V_i}$ and so $[m]_x=[m_i]_x=s(x)$ for all $x\in V_i$.  We conclude that $[m]_x=s(x)$ for all $x\in V_1\cup\cdots\cup V_n$.  If $x\notin V_1\cup\cdots\cup V_n$, then $x\notin K$ and so $s(x)=0$.  But also $[m]_x=\sum_{i=1}^n [m_i\chi_{V_i}]_x=0$ by Lemma~\ref{l:useful}.  Thus $s(x) = [m]_x=s_m(x)$ for all $x\in \mathscr G\skel 0$ and hence $s=\eta_M(m)$.  This concludes the proof $\eta_M$ is an isomorphism.

Next let $(E,p)$ be a $\mathscr G$-sheaf of $\Bbbk$-modules and put $M=\Gamma_c(E,p)$.  We define an isomorphism
\[\xymatrix{\til M\ar[rd]_{p_M}\ar[rr]^{\varepsilon_{(E,p)}} &  &E\ar[ld]^p\\ & \mathscr G\skel 0 &}\] of $\Sh(\Gamma_c(E,p))$ and $(E,p)$ as follows.  Define $\varepsilon_{(E,p)}([s]_x)=s(x)$ for $s\in M$ and $x\in \mathscr G\skel 0$.  This is well defined because if $s\chi_U=s'\chi_U$ for some compact open neighborhood $U$ of $x$, then $s(x) =s'(x)$ by Proposition~\ref{p:bisecaction}.  Also, $p(s(x))=x=p_M([s]_x)$, whence $p\circ\varepsilon_{(E,p)}=p_M$. Clearly, $\varepsilon_{(E,p)}$ restricts to a $\Bbbk$-module homomorphism on each fiber.  Let $g\colon y\to x$ and suppose $U\in \Bis_c(\mathscr G)$ with $g\in U$, then $\varepsilon_{(E,p)}([s]_xg) = \varepsilon_{(E,p)}([s\chi_U]_y)=(s\chi_U)(y) = s(x)g=\varepsilon_{(E,p)}([s]_x)g$ (using Proposition~\ref{p:bisecaction}). It therefore remains to prove that $\varepsilon_{(E,p)}$ is a homeomorphism.

To see that $\varepsilon_{(E,p)}$ is continuous, let $[s]_x\in \til M$ and let $U$ be a neighborhood of $\varepsilon_{(E,p)}([s]_x)=s(x)$. Let $W$ be a compact open neighborhood of $x$ with $s(W)\subseteq U$.  Consider the neighborhood $(W,s)$ of $[s]_x$.  Then, for $[s]_z\in (W,s)$, we have $\varepsilon_{(E,p)}([s_z])=s(z)\in U$.  Thus $\varepsilon_{(E,p)}$ is continuous. As $p,p_M$ are local homeomorphisms, we deduce from $p\circ\varepsilon_{(E,p)}=p_M$ that $\varepsilon_{(E,p)}$ is a local homeomorphism and hence open.  It remains to prove that $\varepsilon_{(E,p)}$ is bijective.

Suppose that $s(x)=\varepsilon_{(E,p)}([s]_x) =\varepsilon_{(E,p)}([t]_x)=t(x)$. Choose a neighborhood $U$ of $s(x)=t(x)$ such that $p|_U$ is a homeomorphism onto its image.  Let $W$ be a compact open neighborhood of $x$ such that both $s(W)\subseteq U$ and $t(W)\subseteq U$.  Then if $z\in W$, we have $p(s(z))=z=p(t(z))$ and $s(z),t(z)\in U$ and hence $s(z)=t(z)$.  Thus $s\chi_W=t\chi_W$ (cf.~Proposition~\ref{p:bisecaction}) and so $[s]_x=[t]_x$.  This yields injectivity of $\varepsilon_{(E,p)}$.  Next let $e\in E_x$. Let $U$ be neighborhood of $e$ such that $p|_U\colon U\to p(U)$ is a homeomorphism.  Let $W$ be a compact open neighborhood of $x$ contained in $p(U)$ and define $s\in \Gamma_c(E,p)$ to agree with $(p|_U)\inv$ on $W$ and be $0$ outside of $W$.  Then $\varepsilon_{(E,p)}([s]_x) = s(x)=e$.  Clearly, $\varepsilon_{(E,p)}$ is natural $(E,p)$.  This completes the proof.
\end{proof}

As a corollary, we recover the main result of~\cite{GroupoidMorita}, and moreover we do not require the Hausdoff assumption.

\begin{Cor}
Let $\mathscr G$ and $\mathscr H$ be Morita equivalent ample groupoids.  Then $\Bbbk \mathscr G$ is Morita equivalent to $\Bbbk\mathscr H$ for any commutative ring with unit $\Bbbk$.
\end{Cor}

By restricting to the case where $\mathscr G\skel 1=\mathscr G\skel 0$, we also have the following folklore result.

\begin{Cor}
Let $X$ be a Hausdorff space with a basis of compact open subsets and $\Bbbk$ a commutative ring with unit.  Let $C_c(X,\Bbbk)$ be the ring of locally constant functions $X\to \Bbbk$ with compact support.  Then the category of sheaves of $\Bbbk$-modules on $X$ is equivalent to the category of unitary $C_c(X,\Bbbk)$-modules.
\end{Cor}

If $\mathscr G$ is a discrete groupoid, then $\mathcal B\mathscr G$ is equivalent to the category $\mathbf{Set}^{\mathscr G^{op}}$ of contravariant functors from $\mathscr G$ to the category of sets~\cite{elephant1,elephant2}.  Therefore, $\mathcal B_{\Bbbk}\mathscr G$ is equivalent to the category $(\modu \Bbbk)^{\mathscr G^{op}}$ of contravariant functors from $\mathscr G$ to $\modu \Bbbk$.  It is well known that $(\modu \Bbbk)^{\mathscr G^{op}}$ is equivalent to $\modu \Bbbk\mathscr G$ when $\mathscr G\skel 0$ is finite~\cite{ringoids} and presumably the following extension is also well known, although the author doesn't know a reference.

\begin{Cor}
Let $\mathscr G$ be a discrete groupoid and $\Bbbk$ a commutative ring with unit.  Then $\modu \Bbbk\mathscr G$ is equivalent to the category $(\modu \Bbbk)^{\mathscr G^{op}}$  of contravariant functors $\mathscr G\to \modu \Bbbk$.  Hence naturally equivalent discrete groupoids have Morita equivalent algebras.
\end{Cor}

%\bibliographystyle{abbrv}

%\bibliography{sgpabb,bibliografia,standard2}
%\bibliography{standard2}

\end{document}